\documentclass[12pt]{article}

\usepackage{fullpage}
\usepackage[normalem]{ulem}

\usepackage{
    amsthm,
    amsmath,
    amssymb,
    appendix,
    comment,
    enumerate,
    graphicx,
    mathrsfs,
    mathtools,
    thm-restate
}

\usepackage[dvipsnames]{xcolor}

\usepackage[capitalize]{cleveref}

\usepackage{tikz}
\usetikzlibrary{arrows}

\newtheorem{theorem}{Theorem}[section]
\newtheorem*{theorem*}{Theorem}
\newtheorem{conjecture}[theorem]{Conjecture}
\newtheorem{lemma}[theorem]{Lemma}

\theoremstyle{remark}
\newtheorem{example}[theorem]{Example}

\providecommand{\Z}{}
\providecommand{\N}{}

\renewcommand{\Z}{\mathbb{Z}}
\renewcommand{\N}{{\mathbb N}}

\newcommand\cD{\mathcal D}

\newcommand\cG{\mathcal G}

\newcommand\cO{\mathcal O}
\newcommand\cP{\mathcal P}

\newcommand\cT{{\mathcal T}}

\newcommand\cV{\mathcal V}

\newcommand{\ld}[2]{\ensuremath{\underline{\delta}\left(#1,#2\right)}}
\newcommand{\ud}[2]{\ensuremath{\overline{\delta}\left(#1,#2\right)}}

\newcommand{\ith}[1]{#1^{\text{th}}}

\title{Achievable Burning Densities of Growing Grids}

\author{
Jordan Barrett\\
Dalhousie University\\
\texttt{jbarrett@dal.ca}
\and 
Karen Gunderson\\
University of Manitoba\\
\texttt{karen.gunderson@umanitoba.ca}
\and
JD Nir\\
Oakland University \\
\texttt{jdnir@oakland.edu}
\and
Pawe\l{} Pra\l{}at\\
Toronto Metropolitan University\\
\texttt{pralat@torontomu.ca}
}

\begin{document}
\maketitle

\begin{abstract}
Graph burning is a discrete-time process on graphs where vertices are sequentially activated and burning vertices cause their neighbours to burn over time. In this work, we focus on a dynamic setting in which the graph grows over time, and at each step we burn vertices in the growing grid $G_n = [-f(n),f(n)]^2$. We investigate the set of achievable burning densities for functions of the form $f(n)=\lceil cn^\alpha\rceil$, where $\alpha \ge 1$ and $c>0$. We show that for $\alpha=1$, the set of achievable densities is $[1/(2c^2),1]$, for $1<\alpha<3/2$, every density in $[0,1]$ is achievable, and for $\alpha=3/2$, the set of achievable densities is $[0,(1+\sqrt{6}c)^{-2}]$.
\end{abstract}

\section{Introduction}\label{sec:intro}

Graph burning is a discrete-time process that models the spread of influence in a network. Vertices are in one of two states: either \emph{burning} or \emph{unburned}. In each round, a burning vertex causes all of its neighbours to become burning and a new \emph{fire source} is chosen: a~vertex whose state is changed to burning regardless of its neighbours and of its previous state. The updates repeat until all vertices are burning. The \emph{burning number} of a graph $G$, denoted $b(G)$, is the minimum number of rounds required to burn all of the vertices of $G$.  

Graph burning first appeared in print in a paper of Alon~\cite{nA92}, motivated by a question of Brandenburg and Scott at Intel, and was formulated as a transmission problem involving a set of processors. It was then independently studied by Bonato, Janssen, and Roshanbin~\cite{BJR14,BJR16,eR16} who, in addition to introducing the name \emph{graph burning}, gave bounds and characterized the burning number for various graph classes. The problem has since received wide attention (e.g.~\cite{BBBCKPR21,BBJRR18,DEP24,LL16,MPR17,MPR18,NT22}), with particular focus given to the so-called Burning Number Conjecture that every connected graph on $n$ vertices requires at most $\lceil \sqrt{n} \rceil$ rounds to burn.

The concept of burning density, first introduced by Bonato, Gunderson, and Shaw~\cite{BGS20}, replaces a static network with a growing sequence of graphs. In this paradigm, there may not be a step in which every vertex is burning. Instead, we consider the proportion of burning vertices to the total number of vertices.

Let $\cG = (G_n, n \geq 1)$ be a sequence of connected graphs with the property that $G_n$ is an induced subgraph of $G_{n+1}$ for all $n \geq 1$. Next, let $\cV = (v_n, n \geq 1)$ be a sequence of vertices such that $v_n \in V(G_n) \cup \{\emptyset\}$ for all $n \geq 1$. We consider the following \emph{burning process on $\cG$} with \emph{activator sequence $\cV$}. Beginning with the null graph $G_0$ and the empty set $B_0$ of burning vertices, for $n \geq 1$, 
\begin{enumerate}
\item add vertices and edges to $G_{n-1}$ to construct $G_n$,
\item set $B_n$ to be the union of $B_{n-1}$ and each vertex in the neighbourhood of $B_{n-1}$ in $G_n$, and finally
\item if $v_n \neq \emptyset$, add $v_n$ to $B_n$.
\end{enumerate}
We think of this as a process in time where, at each time interval, the graph grows, the fire spreads, and a new vertex is burned; we often refer to the process at time $t$ as \emph{turn $t$}. We do not insist on $v_t$ being unburned when it is activated, nor do we insist that $v_t \notin \{ v_1,\dots,v_{t-1}\}$. Likewise, we allow an activator sequence to ``pass'' on turn $t$ by setting $v_t = \emptyset$. The only requirement we insist on is that $v_t \neq \emptyset$ for \emph{at least one} $t$. See Figure~\ref{fig:stop go} for an example of a burning process.

For a graph sequence $\cG$ and activator sequence $\cV$, define the \emph{lower burning density} and the \emph{upper burning density} of $\cV$ in $\cG$, respectively, as 
\[
\ld{\cG}{\cV}
:= 
\liminf_{n \to \infty} \frac{|B_n|}{|V(G_n)|}
\qquad \text{ and } \qquad
\ud{\cG}{\cV}
:= 
\limsup_{n \to \infty} \frac{|B_n|}{|V(G_n)|} \,.
\]
Trivially, $0 \le \underline{\delta}(\cG,\cV) \le \overline{\delta}(\cG,\cV) \le 1$. If $\underline{\delta}(\cG,\cV) = \overline{\delta}(\cG,\cV)$ for a pair $(\cG,\cV)$, then we call this value the \emph{burning density} of $\cV$ in $\cG$ and denote it $\delta(\cG,\cV)$. Note that if an activator sequence $\cV = (v_n, n\ge 1)$ has $v_n = \emptyset$ for all $n$, then $\delta(\cG, \cV) = 0$.  To avoid this trivial case, we assume that activator sequences contain at least some vertices, and call such sequences non-empty.  Finally, write $P(\cG) \subseteq [0,1]$ for the set of obtainable burning densities in $\cG$, defined as 
\begin{align*}
P(\cG) &:= \left\{ \rho \in [0,1] : \text{there is a non-empty activator sequence } \cV \text{ such that } \delta(\cG,\cV) = \rho \right\}.
\end{align*}

\medskip

In this paper we consider sequences $\cG$ of grid graphs on the integer lattice. We write $G = [a,b] \times [c,d]$ when $G$ is the grid graph with vertex set $[a,b] \times [c,d]$ and with edges joining pairs of vertices at $L_1$-distance one from one another. We write $[a,b]^2$ as shorthand for $[a,b] \times [a,b]$. 

Let $f: \N \to \N$ be a strictly increasing function, set $G_n = [-f(n),f(n)]^2$, and define $\cG(f) = \big( G_n, n\geq 1 \big)$. Bonato, Gunderson, and Shaw proved:
\begin{theorem*}[\cite{BGS20}]
Let $f : \N \to \N$ be a non-decreasing function.
\begin{enumerate}[(a)]
\item If $f(n) = \lceil cn \rceil$ for some $c \geq 1$, then $P(\cG(f)) = [1/(2c^2),1]$.
\item If $f(n) = \lceil cn^{3/2} \rceil$ for some $c > 0$, then there exists an activator sequence $\cV$ such that $\underline{\delta}(\cG(f),\cV) > 0$.
\item If $f(n) = \omega(n^{3/2})$, then $P(\cG(f)) = \{0\}$.
\end{enumerate}
\end{theorem*}

Our primary contribution is an extension of their result in which we determine the set of achievable burning densities in $\cG(f)$ for all functions of the form $\lceil cn^\alpha \rceil$, where $\alpha \geq 1$ and $c$ is a positive real number. Note that if $\alpha = 1$ then we insist on $c \geq 1$, as otherwise the grid grows slower than the fire spreads and thus $P(\cG(f)) = \{1\}$ trivially. 
\begin{theorem}\label{thm:main}
Fix positive real numbers $c$ and $\alpha$ and set $f(n) = \lceil cn^\alpha \rceil$.
\begin{enumerate}[(a)]
\item If $\alpha = 1$ and $c \geq 1$, then $P(\cG(f)) = [1/(2c^2),1]$. \label{case:linear}
\item If $\alpha \in (1,3/2)$, then $P(\cG(f)) = [0,1]$. \label{case:wait_and_see}
\item If $\alpha = 3/2$, then $P(\cG(f)) = [0,(1+\sqrt{6} c)^{-2}]$. \label{case:cubic}
\item If $\alpha > 3/2$, then $P(\cG(f)) = \{0\}$. \label{case:big}
\end{enumerate}
\end{theorem}

Note that Bonato, Gunderson, and Shaw proved case (\ref{case:linear}) and proved the strongest possible extension of case (\ref{case:big}) of \cref{thm:main}, namely if $f(n) = \omega(n^{3/2})$, then $P(\cG(f)) = \{0\}$. Our contribution to Theorem~\ref{thm:main} is in proving cases (\ref{case:wait_and_see}) and (\ref{case:cubic}). In fact, we determine $P(\cG(f))$ for a broader class of functions. In the cases (\ref{case:linear}) and (\ref{case:cubic}) of \cref{thm:main} we actually prove a more general result applicable to functions with similar asymptotic behaviour.

\begin{restatable}{theorem}{asymptotic}\label{thm:asymptotic}
Fix positive real numbers $c$ and $\alpha$ and suppose $f : \N \to \N$ is strictly increasing and satisfies $\lim_{n \to \infty} n^{-\alpha} f(n) = c \,.$
\begin{enumerate}[(a)]
\item[(a)] If $\alpha = 1$ and $c \geq 1$, then $P(\cG(f)) = [1/(2c^2),1]$. \label{case:asym_linear}
\item[(c)] If $\alpha = 3/2$, then $P(\cG(f)) = [0,(1+\sqrt{6} c)^{-2}]$. \label{case:asym_cubic}
\end{enumerate}
\end{restatable}

In the remaining case of \cref{thm:main}, case (\ref{case:wait_and_see}), we prove a similar extension, though we require an additional assumption. We say a function $f$ satisfies the \emph{controlled growth requirement} or \emph{has controlled growth} if 
\begin{enumerate}
\item[i)] $f(n+1)>f(n)$ for all sufficiently large $n$, 
\item[ii)] for every function $\epsilon: \N \to \N$ satisfying $\epsilon(n) = o(n)$ we have $f(n+\epsilon(n)) = (1+o(1))f(n)$, and
\item[iii)] for every constant $c > 0$, there is a constant $d > 0$ such that for sufficiently large $n$, $f(n+cn) \geq (1+d)f(n)$.
\end{enumerate}

\begin{restatable}{theorem}{bonus}\label{thm:bonus}
Let $f : \N \to \N$ satisfy the controlled growth requirement and suppose $f(n) = \omega(n)$ and $f(n) = o(n^{3/2})$. Then $P(\cG(f)) = [0,1]$. 
\end{restatable}

\noindent We also provide examples (see~\cref{ex:need smoothness}) demonstrating that \cref{thm:bonus} may not hold if $f$ does not satisfy the controlled growth requirement. 

\bigskip

The remainder of the paper is organized as follows. In Section~\ref{sec:prelim}, we present auxiliary observations that will simplify other proofs. In particular, in Subsection~\ref{sec:aux1} we will show that one may conveniently modify an activator sequence without affecting burning densities. Moreover, restricting to increasing functions $f(n)$ is also justified---see Subsection~\ref{sec:aux2}. This, in turn, shows that one can conveniently modify growth functions---see Subsection~\ref{sec:aux3}. We finish this section with a brief discussion that some form of restriction on the growth function is necessary. Section~\ref{sec:wait and see} is devoted to the proof of Theorem~\ref{thm:bonus} which implies \cref{thm:main} (\ref{case:wait_and_see}). The proof of Theorem~\ref{thm:asymptotic} which implies  Theorem~\ref{thm:main} (\ref{case:cubic}) can be found in Section~\ref{sec:cubic}. We conclude the paper with a few open problems---see Section~\ref{sec:conclusion}.

\section{Preliminary results}\label{sec:prelim}

We present a series of lemmas that, when combined, allow us to modify a burning process in quite substantial ways whilst maintaining the same lower and upper burning densities. Although not all of the results in this section are required to prove the three main theorems, we include them all as they offer a set of tools for analyzing burning processes.

For a burning process $(\cG,\cV)$, write $B_t[i]$ for the subset of $B_t$ that would still be burned if only the $\ith{i}$ vertex in $\cV$ was activated during the process. Note that 
\[
B_t = \bigcup_{i \in [t]} B_t[i] \,,
\]
and also note that $B_t[i] \cap B_t[j]$ is not necessarily empty. Finally, note that $B_t[i]$ is not necessarily the ball of radius $t-i$ centred at $v_i$ in $G_t$. For example (see Figure~\ref{fig:stop go}), if $f(1) = f(2) = 1$, $f(3) = 3$, and vertex $(1,1)$ is activated on turn 1, then $B_1[3]$ is not the ball of radius of 2 centred at $(1,1)$. However, if $f$ is strictly increasing, then $B_t[i]$ is in fact the ball of radius $t-i$ centred at $v_i$ in $G_t$. 

\begin{figure}
\[
\includegraphics[scale=0.45, trim={0cm 5cm 0cm 5cm}]{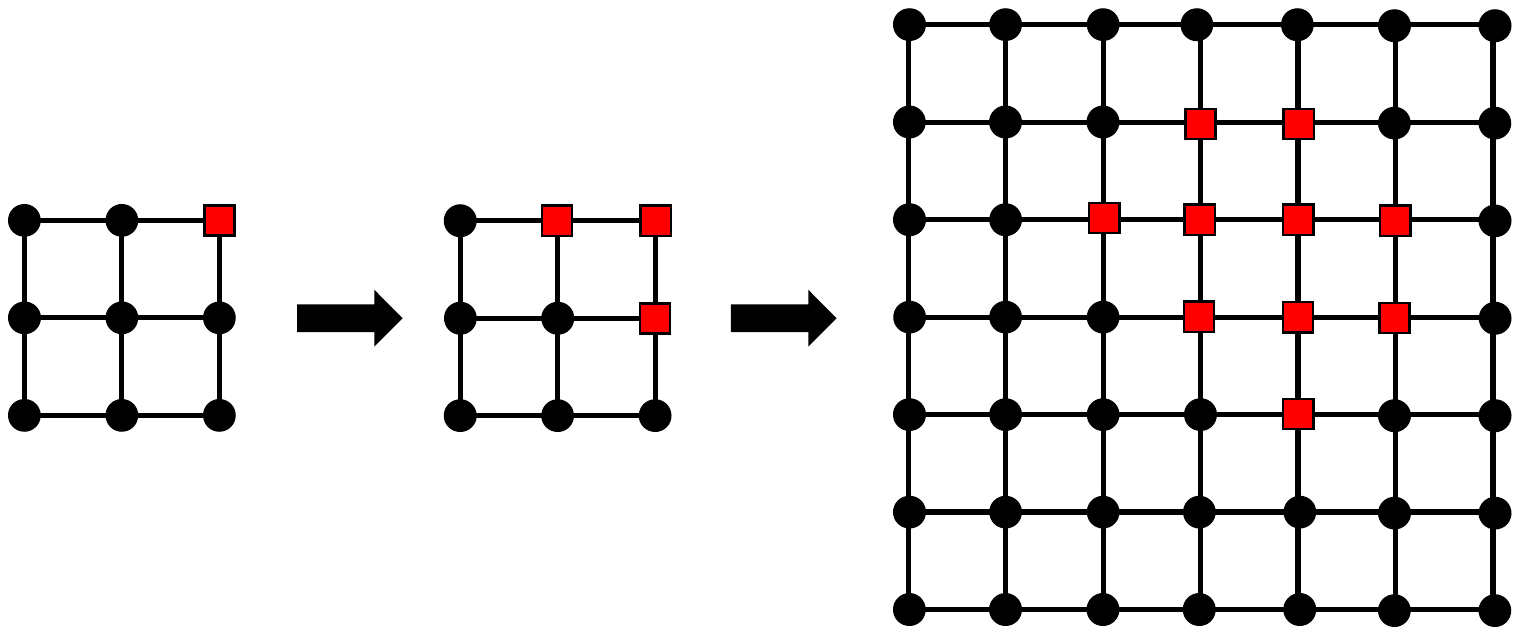}
\]
\caption{The first three turns of a burning process with $f(1)=f(2)=1$ and $f(3)=3$. Vertex $(1,1)$ is activated on turn 1 and no other vertices are activated on subsequent turns. The burned vertices are highlighted as red squares.}\label{fig:stop go}
\end{figure}

\subsection{Key lemmas part 1: manipulating the activator sequence}\label{sec:aux1}

In this first collection of three lemmas, we will show how an activator sequence can be delayed, trimmed at the front, and point-wise perturbed without altering the lower and upper burning densities. For these lemmas, we assume $f$ has controlled growth. 

\begin{lemma}\label{lem:time-shifted}
Fix $k \geq 1$, let $\cO = (v_n, n \geq 1)$ be an activator sequence on $\cG(f)$ and let $\cD = (u_n, n \geq 1)$ be a modified activator sequence with $u_1 = \dots = u_k = \emptyset$ and $u_{k+i} = v_{i}$ for all $i \geq 1$. Then $\ld{\cG(f)}{\cD} = \ld{\cG(f)}{\cO}$ and $\ud{\cG(f)}{\cD} = \ud{\cG(f)}{\cO}$.
\end{lemma}

For ease of readability, we refer to $(\cG(f), \cO)$ and $(\cG(f), \cD)$, respectively, as the \emph{original} and \emph{delayed} burning processes.  

\begin{proof}
Write $B^{\cO}_t$ and $B^{\cD}_t$ for the respective sets of burning vertices in the original and delayed processes after time $t$. Now let $v_i \in \cO$ and note that $v_i = u_{i+k} \in \cD$. Thus, $v_i$ is activated $k$ rounds earlier in the original process than in the delayed process, meaning $v_i$ burns no more in the delayed process than in the original one, i.e., 
\[
B^{\cD}_t[i+k] \subseteq B^{\cO}_t[i] \,.
\]
Moreover, since $B^{\cD}_t[1] = \dots = B^{\cD}_t[k] = \emptyset$ we have that
\[
B^{\cD}_t
= 
\bigcup_{i \in [n]} B^{\cD}_t[i]
\subseteq 
\bigcup_{i \in [n]} B^{\cO}_t[i]
= 
B^{\cO}_t\,,
\]
establishing that $\ld{\cG(f)}{\cD} \leq \ld{\cG(f)}{\cO}$ and $\ud{\cG(f)}{\cD} \leq \ud{\cG(f)}{\cO}$.

For the other direction, we compare the burning sets $B^{\cO}_t[i]$ and $B^{\cD}_{t+k}[i+k]$. If $t=i$, then $B^{\cO}_i[i] = \{v_i\}$ and $B^{\cD}_{i+k}[i+k] = \{u_{i+k}\}$ and these sets are equivalent. Now assume that $B^{\cO}_t[i] \subseteq B^{\cD}_{t+k}[i+k]$ for some arbitrary $t \geq i$ and consider the two burning processes at their respective next steps. For the original process, $B^{\cO}_{t+1}[i]$ is constructed from $B^{\cO}_t[i]$ by adding all adjacent vertices in $G_{t+1}$, then adding $v_{t+1}$. Similarly for the delayed process, $B^{\cD}_{t+k+1}[i+k]$ is constructed from $B^{\cD}_{t+k}[i+k]$ by adding all adjacent vertices in $G_{t+k+1}$, then adding $u_{t+k+1} = v_{t+1}$. Thus, since $G_{t+1} \subseteq G_{t+k+1}$ and $B^{\cO}_t[i] \subseteq B^{\cD}_{t+k}[i+k]$, it follows that $B^{\cO}_{t+1}[i] \subseteq B^{\cD}_{t+k+1}[i+k]$. By induction, we have established that 
\[
B^{\cO}_t[i] \subseteq B^{\cD}_{t+k}[i+k]  \,,
\]
for all $t \geq i$. Therefore,
\[
\frac{|B^{\cD}_{t+k}|}{|V(G_t)|} \geq \frac{|B^{\cO}_{t}|}{|V(G_t)|} \,.
\]
Finally, as $f(n)$ has controlled growth, for $t \to \infty$ and $k$ fixed we have that $|V(G_{t+k})| = (1+o(1))|V(G_{t})|$, implying that 
\[
\liminf_{t \to \infty} \frac{|B^{\cD}_{t+k}|}{|V(G_{t+k})|}
=
\liminf_{t \to \infty} \frac{|B^{\cD}_{t+k}|}{|V(G_{t})|}
\geq 
\liminf_{t \to \infty} \frac{|B^{\cO}_{t}|}{|V(G_t)|} \,,
\]
and similarly for the limit superior. This inequality establishes that $\ld{\cG(f)}{\cD} \geq \ld{\cG(f)}{\cO}$ and $\ud{\cG(f)}{\cD} \geq \ud{\cG(f)}{\cO}$, finishing the proof.
\end{proof}

\begin{lemma}\label{lem:turn1}
Fix $k \geq 1$, let $\cO = (v_n, n \geq 1)$ be an activator sequence on $\cG(f)$ with $v_i \ne \emptyset$ for some $i>k$, and let $\cT = (u_n, n \geq 1)$ be the modified activator sequence with $u_1 = \dots = u_k = \emptyset$ and $u_i = v_i$ for all $i > k$. Then $\underline{\delta}(\cG(f),\cT) = \underline{\delta}(\cG(f),\cO)$ and $\overline{\delta}(\cG(f),\cT) = \overline{\delta}(\cG(f),\cO)$.
\end{lemma}

For ease of readability, we refer to $(\cG(f), \cT)$ as the \emph{trimmed} burning processes.  

\begin{proof}
Write $B^{\cO}_t$ and $B^{\cT}_t$ for the respective burning sets in the original and trimmed processes after time $t$. Note that $\underline{\delta}(\cG(f),\cT) \leq \underline{\delta}(\cG(f),\cO)$ and $\overline{\delta}(\cG(f),\cT) \leq \overline{\delta}(\cG(f),\cO)$ is immediate, as $B^{\cT}_t \subseteq B^{\cO}_t$ for all $t \geq 1$. 

For the other direction, let $\ell > k$ satisfy $V(G_k) \subseteq B^{\cT}_\ell$. Such an $\ell$ exists as there is some time $s \geq k$ when $B^{\cT}_s$ is non-empty, and after time $s+2f(s)+1$ all of $G_s$ (and hence all of $G_k$) is burning. Let $\cD$ be the activator sequence that activates $v_i$ on round $i+\ell$ for all $i \geq 1$ and activates nothing beforehand, i.e., $(\cG(f), \cD)$ is the delayed process in Lemma~\ref{lem:time-shifted} with $k$ replaced by $\ell$. Then, writing $B^{\cD}_t$ for the set of burning vertices in the delayed process after time $t$, we have that 
\begin{align*}
B^{\cD}_t[\ell + i] &\subseteq \bigcup_{j \in [\ell]} B^{\cT}_t[j] \hspace{1cm} \text{for } 1 \leq i \leq k \text{, and} \\
B^{\cD}_t[\ell + i] &\subseteq B^{\cT}_{t}[i] \hspace{1.76cm} \text{for } i > k \,,
\end{align*}
the first containment coming from the fact that $v_1,\dots,v_k \in V(G_k)$ are already burning in the trimmed process before time $\ell+1$, and the second containment due to the fact that $v_i$ is the $\ith{i}$ activator vertex in $\cT$ and the $\ith{(\ell+i)}$ activator vertex in $\cD$. Therefore, $B^{\cD}_t \subseteq B^{\cT}_t$ so $\underline{\delta}(\cG(f),\cT) \ge \underline{\delta}(\cG(f), \cD)$ and $\overline{\delta}(\cG(f),\cT) \ge \overline{\delta}(\cG(f), \cD)$. Furthermore, by Lemma~\ref{lem:time-shifted} we have that $\underline{\delta}(\cG(f),\cD) = \underline{\delta}(\cG(f),\cO)$ and $\overline{\delta}(\cG(f),\cD) = \overline{\delta}(\cG(f),\cO)$. In combination, we get that $\underline{\delta}(\cG(f),\cT) \geq \underline{\delta}(\cG(f),\cO)$ and $\overline{\delta}(\cG(f),\cT) \geq \overline{\delta}(\cG(f),\cO)$, finishing the proof.
\end{proof}

\begin{lemma}\label{lem:position-shifted}
Fix $d \geq 1$ and let $\cO = (v_n,n \geq 1)$ and $\cP = (u_n,n \geq 1)$ be activator sequences on $\cG(f)$ with the property that $\mathrm{dist}_{G_n}(v_n,u_n) \leq d$ for all $n \geq 1$. Then $\underline{\delta}(\cG(f),\cP) = \underline{\delta}(\cG(f),\cO)$ and $\overline{\delta}(\cG(f),\cP) = \overline{\delta}(\cG(f),\cO)$.
\end{lemma}

In this case, we call $(\cG(f), \cP)$ the perturbed process. 

\begin{proof}
By symmetry, it suffices to prove that $\underline{\delta}(\cG(f),\cO) \leq \underline{\delta}(\cG(f),\cP)$ and $\overline{\delta}(\cG(f),\cO) \leq \overline{\delta}(\cG(f),\cP)$. Let $\cD$ be the activator sequence that activates $v_i$ at time $i+d$ for all $i \geq 1$ and activates nothing beforehand, i.e., the delayed process. Then any vertex in $\cD$, at the time of activation, was already burning in the perturbed process by some vertex that (a) is at distance at most $d$, and (b) was activated $d$ turns before. Therefore, $\underline{\delta}(\cG(f),\cD) \leq \underline{\delta}(\cG(f),\cP)$ and $\overline{\delta}(\cG(f),\cD) \leq \overline{\delta}(\cG(f),\cP)$, and once again in combination with Lemma~\ref{lem:time-shifted} this finishes the proof. 
\end{proof}

In summary, Lemmas~\ref{lem:time-shifted}, \ref{lem:turn1}, and \ref{lem:position-shifted} tell us that an activator sequence can be delayed, trimmed, and perturbed without altering the lower and upper burning densities. 

\subsection{Restricting our attention}\label{sec:aux2}

Two of the three main objectives, Theorems~\ref{thm:main} and~\ref{thm:bonus}, involve functions that are \emph{almost} strictly increasing. We take this moment to restrict our attention to functions that are, in fact, strictly increasing. The purpose is to ignore certain corner cases where $f$ might not grow at every step, preventing burning sets to grow outwards in all four directions at a constant rate (e.g., Figure~\ref{fig:stop go}). We provide a simple lemma justifying our restriction.

\begin{lemma}\label{lem:almost -> strictly}
Fix $n_0 \geq 1$ and let $f$ be such that (a) $f(n) \geq n$ for all $n \geq n_0$, and (b) $f(n+1)-f(n) > 0$ for all $n \geq n_0$. Then there is a strictly increasing function $g$ such that (a) $f(n) = g(n)$ for all sufficiently large $n$, and (b) $P(\cG(f)) = P(\cG(g))$.
\end{lemma}

\begin{proof}
Let $\rho \in P(\cG(f))$ and let $\cO$ be an activator sequence for $\cG(f)$ such that $\delta(\cG(f),\cV) = \rho$. Define the activator sequence $\cT$ by replacing the first $n_0$ entries of $\cV$ with $\emptyset$. By Lemma~\ref{lem:turn1}, $\delta(\cG(f), \cO) = \delta(\cG(f), \cT)$. Then define $g$ by $g(n) = n$ for $n < n_0$ and $g(n) = f(n)$ for $n \geq n_0$. Because $\cG(g) = \cG(f)$ on each round after $n_0$, the burning sets of $(\cG(f), \cT)$ are identical to those of $(\cG(g), \cT)$ so $\delta(\cG(g), \cT) = \rho$. By a symmetric argument, if $\cO$ achieves $\delta(\cG(g),\cO) = \rho$ then there exists $\cT$ such that $\delta(\cG(f),\cT) = \rho$.
\end{proof}

An immediate benefit of this restriction is seen in the next lemma.

\begin{lemma}\label{lem:burnt bound}
Let $f$ be strictly increasing. For any activator sequence $(v_n, n\geq 1)$ on $\cG(f)$,
\begin{enumerate}
\item[(a)] $B_t[i]$ is the ball of radius $t-i$ centred at $v_i$ in $\Z \times \Z$, and
\item[(b)] $|B_t| \leq \frac{2t^3+t}{3}$.
\end{enumerate}
\end{lemma}

\begin{proof}
Statement (a) follows immediately from the fact that if $f$ is strictly increasing and $v_i \in G_i$ then all points in $\Z \times \Z$ at distance $k$ from $v_i$ must be in $G_{i+k}$.

For statement (b), we have that $|B_t[i]| = (t-i+1)^2 + (t-i)^2$ if $v_i \neq \emptyset$, and $|B_t[i]| = 0$ otherwise. Hence, 
\begin{align*}
|B_t| 
&\leq
\sum_{i=1}^t |B_t[i]|
=
\sum_{i=1}^t \Big( (t-i+1)^2 + (t-i)^2 \Big)\\
&=
\sum_{j=1}^t \Big( j^2 + (j-1)^2 \Big)
= \sum_{j=1}^t j^2 + \sum_{j=0}^{t-1} j^2 \\
&=
\frac{t(t+1)(2t+1)}{6} + \frac{(t-1)t(2t-1)}{6}\\
&=
\frac{2t^3 + t}{3} \,,
\end{align*}
which finishes the proof of the lemma.
\end{proof}

Note that Lemma~\ref{lem:burnt bound} immediately implies Theorem~\ref{thm:main} case (d), as the number of burning vertices is always $O(t^3)$ and the number of vertices in this case is $\omega(t^3)$. 

\subsection{Key lemmas part 2: manipulating the growth function}\label{sec:aux3}

The first collection of key lemmas were all for the goal of modifying activator sequences. We now turn to lemmas that allow us to modify growth functions. Thanks to Lemma~\ref{lem:almost -> strictly}, we may assume growth functions are strictly increasing in the proofs of these lemmas.

\begin{lemma}\label{lem:f,g same densities}
Let $f,g$ be increasing functions such that at least one of $f,g$ has controlled growth and $\lim_{n \to \infty} f(n)/g(n) = 1$. Then 
\begin{enumerate}
\item[(a)] $f$ and $g$ both have controlled growth, and
\item[(b)] $P(\cG(f)) = P(\cG(g))$. 
\end{enumerate}
\end{lemma}

\begin{proof}
Beginning with (a), since $f$ and $g$ are increasing they both satisfy condition~(i) of the controlled growth requirement. Now suppose $f$ has controlled growth and let $\epsilon: \N \to \N$ satisfy $\epsilon
(n) = o(n)$. Then
\begin{align*}
g(n+\epsilon(n)) 
&= 
(1+o(1)) f(n+\epsilon(n))\\
&=
(1+o(1)) f(n)\\
&=
(1+o(1)) g(n) \,.
\end{align*}
Therefore, $g$ satisfies condition~(ii) of the controlled growth requirement. 

Now fix $c > 0$ and let $d > 0$ be such that for sufficiently large $n$, $f(n+cn) \geq (1+d)f(n)$. Then
\begin{align*}
g(n+cn) 
&= 
(1+o(1)) f(n+cn)\\
&\geq 
(1+o(1)) (1+d)f(n)\\
&=
(1+d+o(1))g(n) \,.
\end{align*}
Thus, for sufficiently large $n$, $g(n+cn) \geq (1+d/2) g(n)$, meaning $g$ satisfies condition~(iii) of the controlled growth requirement. Therefore, if $f$ has controlled growth then so does $g$ and, by a symmetrical argument, if $g$ has controlled growth then so does $f$. 

Continuing with (b), write $\cG(f) = (G^{(f)}_n,n \geq 1)$ and $\cG(g) = (G^{(g)}_n,n \geq 1)$. Let $\rho \in P(\cG(f))$ and let $\cV = (v_n,n\geq 1)$ be an activator sequence on $\cG(f)$ such that $\delta(\cG(f),\cV) = \rho$. Define the activator sequence $\cV' = (v_n', n \geq 1)$ on $\cG(g)$ inductively as follows. First, if $v_1 \in V(G^{(g)}_1)$ then set $v_1'=v_1$, and otherwise set $v_1' = v_0 := (0,0)$. Then, for $n \geq 1$ and $k \geq 0$, if $v_n' = v_{n-k}$, set $v_{n+1}' = v_{n-k+1}$ if $v_{n-k+1} \in V(G^{(g)}_{n+1})$, and otherwise set $v_{n+1}' = \emptyset$. In words, $\cV'$ burns the same vertices as $\cV$ in the same order, except that $\cV'$ ``waits'' until it is able to burn the correct vertices in $\cG(g)$. We claim that $\underline{\delta}\big( \cG(g), \cV' \big) = \overline{\delta}\big( \cG(g), \cV' \big) = \rho$. 

Firstly, note that $\cV'$ is a well defined activator sequence on $\cG(f)$ as well as on $\cG(g)$. Moreover, after any turn $t$, the number of vertices burned by $\cV'$ in $\cG(f)$ is at most the number of vertices burned by $\cV$. Thus, $\overline{\delta}\big( \cG(f), \cV' \big) \leq \delta\big( \cG(f), \cV \big) = \rho$. Next, since $|V(G^{(g)}_n)| = (1+o(1))|V(G^{(f)}_n)|$, and since the sets of burning vertices in $(\cG(f),\cV')$ and in $(\cG(g),\cV')$ are identical at all times, we get that 
\[
\overline{\delta}\big( \cG(g), \cV' \big) = \overline{\delta}\big( \cG(f), \cV' \big) \leq \delta\big( \cG(f), \cV \big) = \rho \,.
\]

We are left to show that $\underline{\delta}\big( \cG(g), \cV' \big) \geq \rho$. For each $n \geq 1$, define $\epsilon(n)$ to be the smallest non-negative integer that satisfies $g(n+\epsilon(n)) \geq f(n)$. We claim that $\epsilon(n) = o(n)$. Suppose, to the contrary, that there exists a constant $c > 0$ and an increasing sequence $(n_i,i\geq 1)$ such that $\epsilon(n_i) \geq cn_i$ for all $i \geq 1$. Then, for each $i \geq 1$, by the minimality of $\epsilon$ we have $g(n_i+\epsilon(n_i)-1) < f(n_i)$. By condition~(ii) of the controlled growth requirement, this inequality implies that $g(n_i+\epsilon(n_i)) = (1+o(1)) f(n_i)$, implying further that 
\[
g(n_i + cn_i) \leq g(n_i+\epsilon(n_i)) = (1+o(1))f(n_i) = (1+o(1))g(n_i) \,.
\] 
However, by condition~(iii) of the controlled growth requirement, we can find a constant $d > 0$ such that for sufficiently large $n$, $g(n + cn) \geq (1+d)g(n)$. Therefore, the assumption that $\epsilon(n) \neq o(n)$ leads to a contradiction, and thus $\epsilon(n) = o(n)$.

Lastly, by the definition of $\epsilon$, the number of unique activator vertices in $(v_i,i \in [n]) \subseteq \cV$ that are not in $(v_i',i \in [n])$ is at most
\[
\epsilon_{\max}(i) := \max_{i \in [n]} \epsilon(i) \,,
\]
and thus the number of burning vertices in $(\cG(g),\cV')$ by the end of turn $t$ is at least the number of burning vertices in $(\cG(f),\cV)$ by the end of turn $t-\epsilon_{\max}(t) = (1-o(1))t$. Letting $B_t$ and $B_t'$ be the respective number of burning vertices in $(\cG(f),\cV)$ and $(\cG(g),\cV')$ by the end of turn $t$, we have that 
\begin{align*}
\frac{|B_t'|}{|V(G^{(g)}_t)|}
&\geq
\frac{|B_{t-\epsilon_{\max}(t)}|}{|V(G^{(g)}_t)|}\\
&= 
(1+o(1))\frac{|B_{t-\epsilon_{\max}(t)}|}{|V(G^{(f)}_t)|}\\
&= 
(1+o(1))\frac{|B_{t-\epsilon_{\max}(t)}|}{|V(G^{(f)}_{t-\epsilon_{\max}(t)})|} \,,
\end{align*}
the last equality holding since $f$ satisfies condition~(i) of the controlled growth requirement. Therefore,
\[
\liminf_{t \to \infty} \frac{|B_t'|}{|V(G^{(g)}_t)|}
\geq
\liminf_{t \to \infty} \frac{|B_t|}{|V(G^{(f)}_t)|} \,,
\]
implying $\underline{\delta}\big( \cG(g), \cV' \big) \geq \delta\big( \cG(f), \cV \big) = \rho$.

Therefore, $\delta \big( \cG(g), \cV' \big)$ is well defined and equals $\rho$. By symmetry, we have shown that $P(\cG(f)) = P(\cG(g))$, and this concludes the proof.  
\end{proof}

The next lemma is our key to converting an activator sequence $\cV$ on $\cG(f)$ with $\delta(\cG(f),\cV) = \rho$ into a new activator sequence $\cV'$ on $\cG(f)$ with $\delta(\cG(f),\cV') = c^2 \cdot \rho$ for any $c \in (0,1)$.

\begin{lemma} \label{lem:f,g frac density}
Let $f,g : \N \to \N$ be strictly increasing functions such that $f \leq g$ (point-wise) and $\lim_{n\to \infty} f(n)/g(n) = c$ for some $c \in (0,1)$. Then $\rho \in P(\cG(f))$ implies $c^2 \rho \in P(\cG(g))$. 
\end{lemma}

\begin{proof}
Let $\cV$ be an activator sequence on $\cG(f)$ such that $\delta(\cG(f),\cV) = \rho$ and consider $\cV$ as an activator sequence on $\cG(g)$. Then $\cV$ is well defined as $[-f(n),f(n)]^2 \subseteq [-g(n),g(n)]^2$ for all $n \geq 1$. Furthermore, as $f$ and $g$ are both strictly increasing, the set of burning vertices $B_t$ by the end of turn $t$ is identical in $(\cG(f),\cV)$ and in $(\cG(g),\cV)$. Therefore, writing $\cG(f) = (G^{(f)}_n,n \geq 1)$ and $\cG(g) = (G^{(g)}_n,n \geq 1)$, we have 
\begin{align*}
\frac{|B_t|}{|V(G^{(f)}_t)|}
&=
\frac{|B_t|}{(2f(t)+1)^2}\\
&=
(1+o(1)) \frac{1}{c^2} \frac{|B_t|}{(2g(t)+1)^2}\\
&=
(1+o(1)) \frac{1}{c^2} \frac{|B_t|}{|V(G^{(g)}_t)|} \,,
\end{align*}
implying that 
\[
\delta(\cG(f),\cV) = c^2 \delta(\cG(g),\cV) = c^2 \rho \,,
\]
and this concludes the proof. 
\end{proof}

\subsection{The need for controlled growth}\label{sec:aux4}

We provide an example showing that at least \emph{some} form of restriction on the growth of $f$ is necessary for attaining burning densities on $\cG(f)$.

\begin{example}\label{ex:need smoothness}
Let $g(n) = \lceil n^{4/3} \rceil$ and define $f:\N \to \N$ as $f(1) = g(1) =1$ and 
\[
f(n) = 
\begin{cases}
g(n) & \text{if } n = 2^k \text{ for some } k \in \N \text{, and}\\
f(n-1) & \text{otherwise.}
\end{cases}
\]
Fix $n$ and let $k$ be the largest integer such that $2^k \le n$. Then
\[ 
\lceil n^{4/3} \rceil = g(n) \ge f(n) = (2^k)^{4/3} = \frac{1}{2^{4/3}} (2^{k+1})^{4/3} \ge  \frac{1}{2^{4/3}} n^{4/3},
\]
so $f(n) = \Theta\left( n^{4/3} \right)$ and thus $f(n) = \omega(n)$ and $f(n) = o(n^{3/2})$. However, for any activator sequence $\cV$ in $\cG(f)$, there are infinitely many turns $t$ with the property that
\[
\frac{|B_t|}{|V(G_t)|} 
\geq \frac{(2^{4/3})^2 |B_{t}|}{|V(G_{t+1})|} + o(1) 
\geq \frac{(2^{4/3})^2 |B_{t+1}|}{4 |V(G_{t+1})|} + o(1)
= \frac{2^{2/3} |B_{t+1}|}{|V(G_{t+1})|} + o(1)\,,
\]
the second inequality coming from the fact that $|B_{t+1}| \leq 4|B_t|+1$, as every vertex burning by turn $t$ can burn at most 4 other vertices in turn $t+1$. Therefore, either $\frac{|B_t|}{|V(G_t)|}$ diverges or converges to 0, and in either case we do not have $P(\cG(f)) = [0,1]$.
\end{example}

As \cref{ex:need smoothness} demonstrates, without condition~(ii) of the controlled growth requirement, $f$ may have sudden ``jumps'' corresponding to $|B_t|/|V(G_t)|$ fluctuating infinitely often. However, we do not claim that the controlled growth requirement is necessary for Theorem~\ref{thm:bonus}, Lemma~\ref{lem:f,g same densities}, or Lemma~\ref{lem:f,g frac density}. We discuss this topic further in the concluding section of the paper. 

\section{Proof of Theorem~\ref{thm:bonus}}\label{sec:wait and see}

We will make use of the following result of Mitsche, Pra\l{}at, and Roshanbin~\cite{MPR17} regarding the burning number of fixed grid graphs. Note that this result has since been generalized to higher dimensions in~\cite{BC25}.

\begin{theorem}[Mitsche, Pra\l{}at, Roshanbin, \cite{MPR17}] \label{thm:MPR}
\[b([m] \times [n]) = \begin{cases} (1+o(1))\sqrt[3]{3mn/2}, & n \ge m = \omega(\sqrt{n})\\ \Theta(\sqrt{n}), & m = O(\sqrt{n}) \end{cases} \]
\end{theorem}

We begin the proof of \cref{thm:bonus} by demonstrating how to achieve density 1. 

\begin{lemma}\label{lem:(d) upper}
Let $f : \N \to \N$ satisfy the controlled growth requirement and $f(n) = o(n^{3/2})$. Then there exists an activator sequence $\cV$ on $\cG(f)$ such that $\delta(\cG(f),\cV) = 1$. 
\end{lemma}
\begin{proof}
Regardless of which vertices were activated up to step $t$, Theorem~\ref{thm:MPR} tells us we can burn all of the vertices of $G_t$ in at most $\Theta(f(t)^{2/3})$ steps. This fact motivates the following strategy.
\begin{enumerate}
\item Write $(t_i,i\geq 1)$ for the sequence defined inductively as $t_1 = 1$ and $t_{i+1} = t_i + \tau_i$ where $\tau_i$ is the length of time required to burn $G_{t_i}$ as per Theorem~\ref{thm:MPR}. 
\item For each $i \geq 1$, choose some activator sequence $\cV_i = (v_{t}, t_i < t \leq t_{i+1})$ that burns all of $G_{t_i}$. 
\item Let $\cV$ be the concatenation of $(\cV_i),i \geq 1$. 
\end{enumerate} 
This $\cV$ gives 
\[
\frac{|B_{t_{i+1}}|}{|V(G_{t_{i+1}})|} 
\geq 
\frac{|V(G_{t_{i}})|}{|V(G_{t_{i+1}})|} 
= 
\frac{(2f(t_i)+1)^2}{(2f(t_i + \tau_i)+1)^2}
=
(1+o(1)) \left(\frac{f(t_i)}{f(t_i+\tau_i)}\right)^2 \,.
\]
By Theorem~\ref{thm:MPR}, combined with the fact that $f(n) = o(n^{3/2})$, we know that $\tau_i = \Theta(f(t_i)^{2/3}) = o(t_i)$. Thus, as $f$ satisfies condition~(i) of the controlled growth requirement, we have that
\[
\frac{|B_{t_{i+1}}|}{|V(G_{t_{i+1}})|}  
\geq 
(1+o(1)) \left(\frac{f(t_i)}{(1+o(1))f(t_i)}\right)^2
=
(1+o(1)),
\]
and so 
\[
\lim_{i \to \infty} \frac{|B_{t_{i}}|}{|V(G_{t_{i}})|} = 1 \,.
\]
This proves that $\cV$ contains a subsequence on which the burning density converges to $1$. Finally, for any $t_i < t < t_{i+1}$ we have 
\[
\frac{|B_t|}{|V(G_t)|} \geq \frac{|V(G_{t_{i-1}})|}{|V(G_{t_{i+1}})|} 
\]
and we also have 
\[
|V(G_{t_{i-1}})| 
= 
(2f(t_{i-1})+1)^2
=
(2(f(t_{i+1})-\tau_i-\tau_{i+1})+1)^2
=
(2f(t_{i+1})+1)^2 - O(f(t_{i+1})^{5/3}) \,,
\]
meaning 
\[
\frac{|B_t|}{|V(G_t)|} \geq (1-o(1)) \frac{|V(G_{t_{i+1}})|}{|V(G_{t_{i+1}})|} \to 1 \,,
\]
and this completes the proof. 
\end{proof}

We are now ready to prove \cref{thm:bonus} as a straightforward consequence of Lemmas~\ref{lem:f,g frac density} and~\ref{lem:(d) upper}.

\begin{proof}[Proof of \cref{thm:bonus}]
Let $f: \N \to \N$ be a strictly increasing function with controlled growth such that $f(n) = \omega(n)$ and $f(n) = o(n^{3/2})$. Then $|V(G_n)| = \omega(n^2)$, so the activator sequence $\cV = (v_1,n \ge 1)$, consisting of a single repeating vertex, satisfies $|B_n| = O(n^2)$ and therefore $\delta(\cG(f),\cV) = 0$. 

Now fix $\rho \in (0,1]$ and let $g(n) := \lceil \sqrt{\rho} f(n) \rceil$. Note that, although $g$ satisfies conditions (ii) and (iii) of the controlled growth requirement since $\sqrt{\rho} f(n) \leq g(n) \leq \sqrt{\rho} f(n) + 1$, it is not necessarily true that $g(n+1)>g(n)$ for all sufficiently large $n$. For example, if $f(n+1) = f(n)+1$ for infinitely many $n$, then $g(n+1) = g(n)$ infinitely often. Thus, we define $g^+(n)$ recursively as $g^+(0) = g(0)$ and $g^+(n) := \max\left\{ g(n), g^+(n-1) + 1 \right\}$. Since $g(n) = \omega(n)$, we have that $g^+(n) \leq g(n) + n = (1+o(1))g(n)$. Finally, by \cref{lem:(d) upper}, $1 \in P(\cG(g^+))$, implying by Lemma~\ref{lem:f,g frac density} that $\rho \in P(\cG(f))$, and this concludes the proof. 
\end{proof}

\subsection{Proof of \cref{thm:main} (\ref{case:wait_and_see})}

In this section we use the following lemma to derive \cref{thm:main} (\ref{case:wait_and_see}) from \cref{thm:bonus}. 

\begin{lemma}\label{lem:poly_has_controlled_growth}
Fix $\alpha > 1$ and $c > 0$. Then $f(n) = \lceil cn^\alpha \rceil$ satisfies the controlled growth requirement.
\end{lemma}

\begin{proof}
It is enough to show $g(n) = cn^\alpha$ satisfies the controlled growth requirement as $g(n) \leq f(n) \leq g(n)+1$. Condition (i) follows immediately since, by the mean value theorem, $c(n+1)^\alpha - cn^\alpha \in [\alpha c n^{\alpha-1}, \alpha c (n+1)^{\alpha-1}]$, and $\alpha c n^{\alpha - 1} \geq 1$ for sufficiently large $n$. 

Moving on to condition (ii), we have that
$$
g(n+\epsilon(n)) 
= 
c(n+\epsilon(n))^\alpha
=
(1+o(1)) cn^\alpha
=
(1+o(1)) g(n)\,,
$$
and so $g$ satisfies condition~(ii) of the controlled growth requirement. 

For condition~(iii), fix a constant $b > 0$. Then, as $\alpha \geq 1$,
\[
g(n+bn) 
=
c(n+bn)^\alpha
= cn^\alpha (1+b)^\alpha
\geq 
cn^\alpha + bcn^\alpha
=
(1+b)g(n) \,,
\]
meaning $g$ satisfies condition~(iii) of the controlled growth requirement. 
\end{proof}

\begin{proof}[Proof of Theorem~\ref{thm:main} (\ref{case:wait_and_see})]
Let $c > 0$ and $\alpha \in (1, 3/2)$ and let $f(n) = \lceil cn^\alpha \rceil$. Then $f(n) = \omega(n)$ and $f(n) = o(n^{3/2})$ and, by Lemma~\ref{lem:poly_has_controlled_growth}, $f$ has controlled growth. Thus, by Theorem~\ref{thm:bonus}, $P(\cG(f)) = [0,1]$. 
\end{proof}

\section{Proof of Theorem~\ref{thm:main} (\ref{case:cubic})} \label{sec:cubic}

For the duration of Section~\ref{sec:cubic}, fix $c > 0$ and let $f(n) = \lceil cn^{3/2} \rceil$. We begin by showing that $(1+\sqrt{6}c)^{-2}$ is a tight bound for the maximum density in $P(\cG(f))$.
\begin{lemma}\label{lem:3/2main}
Fix $c > 0$ and let $f(n) = \lceil cn^{3/2} \rceil$. Then
\begin{enumerate}[(a)]
\item there exists an activator sequence $\cV$ in $\cG(f)$ such that $\delta(\cG(f),\cV) = (1+\sqrt{6}c)^{-2}$, and 
\item for all activator sequences $\cV$ in $\cG(f)$ we have $\overline{\delta}(\cG(f),\cV) \leq (1+\sqrt{6}c)^{-2}$.
\end{enumerate}
\end{lemma}

\subsection{Upper bound}\label{sec:upperbound}

We begin with the proof of Lemma~\ref{lem:3/2main} (b) as it is more straightforward. 

\begin{proof}[Proof of \cref{lem:3/2main}~(b)]
Let $\cV = (v_n, n \geq 1)$ be an activator sequence on $\cG(f)$, fix $\varepsilon \in (0,1)$, and let $t = (1+\varepsilon)n$. Define $B_t^{(\text{old})} \subseteq B_t$ to be the set of vertices in $\bigcup_{i=1}^n B_t[i]$ and define $B_t^{(\text{new})} \subseteq B_t$ to be the set of vertices in $\bigcup_{i=n+1}^t B_t[i]$. Note that $B_t = B_t^{(\text{old})} \cup B_t^{(\text{new})}$ and so $|B_t| \leq |B_t^{(\text{old})}| + |B_t^{(\text{new})}|$. 

To bound $|B_t^{(\text{old})}|$, note that the centres of balls $B_t[1],\dots,B_t[n]$ are all in $V(G_n)$, meaning $\bigcup_{i=1}^n B_t[i]$ covers at most every vertex in $V(G_n) \subseteq V(G_t)$ plus every vertex of distance at most $t-n = \varepsilon n$ from $V(G_n)$. Since the length of the perimeter of $G_n$ is $\Theta( n^{3/2})$, we get that
\[
|B_t^{(\text{old})}| \leq |V(G_n)| + O(n^{5/2}) = 4c^2n^3 + O(n^{5/2}) \,.
\]
For $|B_t^{(\text{new})}|$, we have the immediate bound
\[
|B_t^{(\text{new})}| 
\leq 
\sum_{i=n+1}^t |B_t[i]|
=
\frac{2(t-n)^3 + (t-n)}{3}
=
\frac{2(\varepsilon n)^3}{3} + O(n) \,.
\]

Therefore, 
\begin{align*}
|B_t|
&\leq
\left( 4c^2 + \frac{2 \varepsilon^3}{3}\right) \, n^3 + O(n^{5/2}) \,,
\end{align*}
and so 
$$
\frac {|B_t|}{|V(G_t)|} \le f(c,\varepsilon) + O(n^{-1/2}), \qquad \text{with} \ f(c,\varepsilon) = \frac{4c^2 + \frac{2 \varepsilon^3}{3}}{4c^2(1+\varepsilon)^3}. 
$$
Since
$$
f'_\varepsilon(c,\varepsilon) = \frac {-6c^2 + \varepsilon^2}{2c^2(1 + \varepsilon)^4},
$$
letting $\varepsilon = \sqrt{6}c$ to minimize the upper bound for $|B_t| / |V(G_t)|$, we conclude that 
\begin{align*}
\overline{\delta}(\cG(f),\cV) 
\leq
\frac{4c^2 + \frac{2 \varepsilon^3}{3}}{4c^2(1+\varepsilon)^3} 
=
\frac{1 + \frac{2 (\sqrt{6} c)^3}{12 c^2}}{(1+\sqrt{6} c)^3} 
=
\frac{1 + \sqrt{6} c}{(1+\sqrt{6} c)^3} 
=
\frac{1}{(1+\sqrt{6} c)^2}, 
\end{align*}
and the proof is finished.
\end{proof}

\subsection{Lower bound}\label{sec:lowerbound}

In this section we prove Lemma~\ref{lem:3/2main} (a) by describing the construction of an activator sequence $\cV$ in $\cG(f)$ such that $\delta(\cG(f),\cV) = (1+\sqrt{6}c)^{-2}$. As suggested during the proof of Lemma~\ref{lem:3/2main} (b), our strategy ensures the balls surrounding newly activated vertices stay mostly disjoint while the balls surrounding older activated vertices, which must necessarily overlap, have burned almost all vertices that were added up to the point that these older vertices were activated.

At the heart of the strategy is an explicit description of how to burn a rectangle of a specified height and width using a set number of activation vertices. While describing this strategy, we temporarily describe activator vertices using real numbers which may not be integers. We then describe how to recursively apply this rectangle burning strategy to burn $\cG(f)$. Finally, we analyze the burning density by estimating the contributions of the new and old activator vertices. 

Let us justify using real numbers as activator vertices. For an increasing function $f : \N \to \N$, consider the following burning process on the sequence of real compact spaces $(S_n(f), n \geq 1)$, where $S_n(f) = [-f(n), f(n)+1]^2$, and a sequence of activator points $((x_n,y_n), n \geq 1)$ where $(x_n,y_n) \in S_n(f) \cup \{\emptyset\}$. Note that, exclusively in this description, we are writing $[-f(n), f(n)+1]^2$ as the real compact space and not the integer lattice. For each time $t$,
\begin{enumerate}
    \item the space grows from $S_{t-1}(f)$ to $S_t(f)$,
    \item if every point in the unit square $[a, a+1] \times [b, b+1]$ is burning then every point in the four squares $[a-1, a] \times [b, b+1]$, $[a+1, a+2] \times [b, b+1]$, $[a, a+1] \times [b-1, b]$, and $[a, a+1] \times [b+1, b+2]$ changes its state to burning, and finally
    \item every point in the unit square $[x_t, x_{t}+1] \times [y_t, y_{t}+1]$ changes its state to burning.
\end{enumerate}
If a unit square is not fully in $S_t$, we simply do not burn any of the square. The burning density is defined as the limiting fraction of the space that is burning. 

It is clear that, by choosing $((x_n,y_n), n \geq 1)$ as integer pairs, this burning process is identical to the burning process on the integer lattice. Moreover, similar to how we proved Lemma~\ref{lem:time-shifted}, we can see that the burning density on $\cV = ((x_n,y_n), n \geq 1)$ is identical to the burning density on $\cV^* = ((\lfloor x_n \rfloor, \lfloor y_n \rfloor), n \geq 1)$: on the one hand we can delay $\cV$ by one turn and then everything burning due to $\cV$ is also burning due to $\cV^*$, and on the other hand we can delay $\cV^*$ by one turn for the inverse effect. Thus, we are justified in choosing real coordinates instead of integer coordinates in our coming strategy. 

\subsubsection{Burning a rectangle}

Let $R$ be a rectangle with width $w$ and height $\ell$. Such a rectangle can be nearly covered with diamonds of (long) radius $1$ (and so of area 2) centred at the points $(i,j)$ with $0 \leq i < w$, $0 \leq j < \ell$, and $i+j \equiv 1 \mod 2$. See Figure~\ref{fig:tiling} for a depiction of this near covering. Suppose now that we are limited to at most $t$ diamonds and we wish to recreate this tiling pattern in $R$. In this case, we choose the unique radius $r$ such that $t = w\ell/2r^2$ and tile $R$ via diamonds of radius $r$ centred at points $(r \cdot i, r \cdot j)$ with $0 \leq i < w/r$, $0 \leq j < \ell/r$ and $i+j \equiv 1 \mod 2$. Note that the number of diamonds required for this tiling is $\lfloor w/r \rfloor \cdot \lfloor \ell/r \rfloor / 2 \leq w\ell/2r^2 = t$. 
 
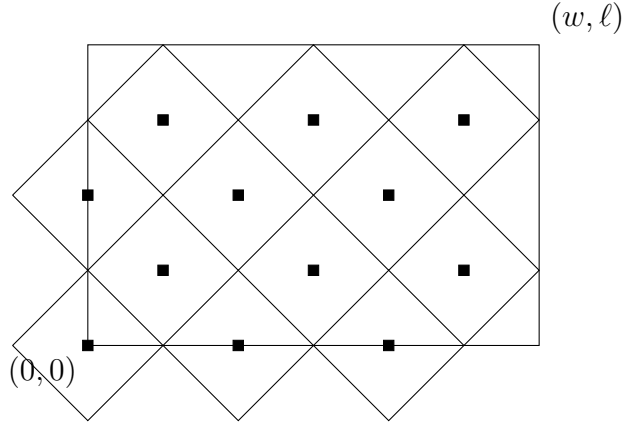
\begin{figure}
\begin{center}

\begin{tikzpicture}
\draw (0,0) rectangle (6,4);
\node[anchor=north east] at (0,0) {$(0,0)$};
\node[anchor=south west] at (6,4) {$(w,\ell)$};

\node[draw, scale=0.5, fill] at (0,0) {}; 
\draw (-1, 0) -- (0,-1) -- (1,0) -- (0,1) -- (-1,0);

\node[draw, scale=0.5, fill] at (2,0) {}; 
\draw (1,0) -- (2,-1) -- (3,0) -- (2,1) -- (1,0);

\node[draw, scale=0.5, fill] at (4,0) {}; 
\draw (3,0) -- (4,-1) -- (5,0) -- (4,1) -- (3,0);

\node[draw, scale=0.5, fill] at (1,1) {}; 
\node[draw, scale=0.5, fill] at (3,1) {}; 
\node[draw, scale=0.5, fill] at (5,1) {}; 
\draw (5,2) -- (6,1) -- (5,0); 

\node[draw, scale=0.5, fill] at (0,2) {}; 
\draw (-1,2) -- (0,1) -- (1,2) -- (0,3) -- (-1,2);

\node[draw, scale=0.5, fill] at (2,2) {}; 
\draw (1,2) -- (2,1) -- (3,2) -- (2,3) -- (1,2);

\node[draw, scale=0.5, fill] at (4,2) {}; 
\draw (3,2) -- (4,1) -- (5,2) -- (4,3) -- (3,2);

\node[draw, scale=0.5, fill] at (1,3) {}; 
\draw (0,3) -- (1,4) -- (2,3);

\node[draw, scale=0.5, fill] at (3,3) {}; 
\draw (2,3) -- (3,4) -- (4,3);

\node[draw, scale=0.5, fill] at (5,3) {}; 
\draw (4,3) -- (5,4) -- (6,3) -- (5,2);
\end{tikzpicture}

\end{center}
\caption{A near-cover of a $w$ by $\ell$ rectangle with diamond tiles centred at points $(i, j)$ for $0 \leq i \leq w$, $0 \leq j \leq \ell$ such that $i+j \equiv 1 \mod 2$. Note that smooth-edged diamonds are drawn for simplicity. Each diamond is in fact a collection of burning squares forming a diamond with stair-case edges.}\label{fig:tiling}
\end{figure}

We now translate this near-covering via diamonds to a strategy for burning a rectangular lattice graph with a limited number of activator vertices. Note that, in the coming lemma, we will not assume that vertices are activated on every turn. The reason for this is that our ultimate strategy requires burning 4 rectangles at the same time, meaning we will need an efficient strategy for burning a rectangle without burning it on every turn. 

\begin{lemma}\label{lem:rectangle strategy}
Let $G = [0, w] \times [0, \ell]$, fix $t > 0$, let $r = \sqrt{w \ell/2t}$, and let $S$ be the set of points $(r \cdot i, r \cdot j)$ such that $0 \leq i \leq w/r$, $0 \leq j \leq \ell/r$, and $i+j \equiv 1 \mod 2$. Now suppose $\cV$ is an activator sequence for $G$ that (a) activates all of $S$ in an arbitrary order, (b) does not activate any vertex outside of $S$, and (c) ends on turn $\tau \geq t$. Then 
\begin{enumerate}[(a)]
\item on turns $n$ with $1 \leq n \leq r$, the burning balls are all disjoint, and 
\item on turns $n$ with $n \geq r+\tau$, there is a $1+O\left( \frac{r(w+\ell)}{w\ell} \right)$ fraction of $V(G)$ that is burning. 
\end{enumerate}
\end{lemma}

\begin{proof}
Starting with (a), the minimum distance between any two vertices in $\cV$ is $2r$. Moreover, at the end of turn $\lfloor r \rfloor$, all of the balls have radius at most $\lfloor r \rfloor - 1$. Thus, the balls are all disjoint. 

Continuing with (b), after turn $\lceil r \rceil + \tau$ all of the balls have radius at least $r$. For any $(x, y) \in [0, w-r] \times [0, \ell-r]$, let $i, j$ be the unique non-negative integers satisfying $r \cdot i \leq x < r \cdot (i+1)$ and $r \cdot j \leq y < r \cdot (j+1)$. 

\bigskip
\noindent
\textbf{Case 1:}
Suppose that $i+j \equiv 0 \mod 2$. Then $(r \cdot i, r \cdot j)$ and $(r \cdot (i+1), r \cdot (j+1))$ are both activator vertices. Moreover, in the $L^1$ norm, 
\[
||(x,y) - (r \cdot i, r \cdot j)||_1 + ||(x,y) - (r \cdot (i+1), r \cdot (j+1))||_1 = 2r, 
\]
meaning $(x, y)$ is at most distance $r$ away from one of the two activator vertices. 

\bigskip
\noindent
\textbf{Case 2:}
Suppose that $i+j \equiv 1 \mod 2$. Then $(r \cdot i, r \cdot (j+1))$ and $(r \cdot (i+1), r \cdot j)$ are both activator vertices. Moreover, in the $L^1$ norm, 
\[
||(x,y) - (r \cdot i, r \cdot (j+1))||_1 + ||(x,y) - (r \cdot (i+1), r \cdot j)||_1 = 2r, 
\]
meaning $(x, y)$ is at most distance $r$ away from one of the two activator vertices. 

Therefore, $[0, w-r] \times [0, \ell-r]$ is covered after time $r + \tau$, meaning at most $r \ell + r w + r^2$ vertices are not covered. Since 
\[
\frac{r \ell + r w + r^2}{w\ell} = O\left(\frac{r(w+\ell)}{w\ell}\right) \,,
\]
this completes the proof.
\end{proof}

\subsubsection{Recursively burning with rectangles}

\begin{figure}[h]
\definecolor{jordan}{RGB}{96,130,255}

\begin{center}\begin{tikzpicture}
\draw (0,0) rectangle (7.5,7.5);
\draw[fill=jordan] (0,1.75) rectangle (1.75,5.75);
\draw[fill=jordan] (1.75,0) rectangle (5.75,1.75);
\draw[fill=jordan] (5.75,1.75) rectangle (7.5,5.75);
\draw[fill=jordan] (1.75,5.75) rectangle (5.75,7.5);
\draw[fill=gray] (1.75,1.75) rectangle (5.75,5.75);

\node at (0.875,3.75) {\large $L(n,k)$};
\node at (6.625,3.75) {\large $R(n,k)$};
\node at (3.75,0.875) {\large $D(n,k)$};
\node at (3.75,6.625) {\large $U(n,k)$};

\draw[dashed,thick] (1.75,7.5) -- (3.25,8.25);
\draw[dashed,thick] (5.75,7.5) -- (4.25,8.25);
\node at (3.75,8.5) {$w(n,k)$};

\draw[dashed,thick] (5.75,5.75) -- (6.1,6.5);
\draw[dashed,thick] (5.75,7.5) -- (6.1,6.75);
\node[anchor=west] at (6,6.625) {$\ell(n,k)$};
\end{tikzpicture}\end{center}

\caption{The gray center square represents $G_n$ and the whole square represents $G_{n+k}$. Each of the four labelled rectangles have two sides of length $w(n,k) = 2c n^{3/2}$ and two sides of length $\ell(n,k) = c \left( (n+k)^{3/2} - n^{3/2} \right)$.\label{fig:rectangles}}
\end{figure}
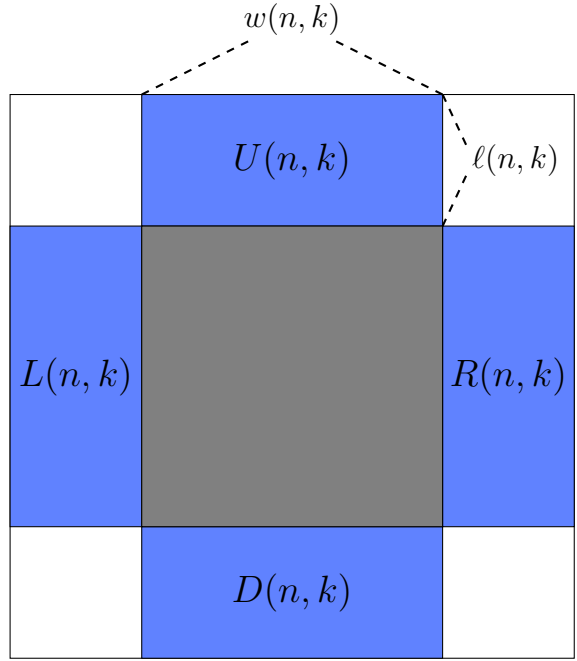

Given $\cG(f) = (G_n,n\geq 1)$, let $U(n,k),L(n,k),R(n,k)$, and $D(n,k)$ be the four rectangles defined in Figure~\ref{fig:rectangles}. Write
\[ \ell(n,k) := f(n+k) - f(n) = c \left( (n+k)^{3/2} - n^{3/2} \right)\] and
\[ w(n,k) := 2f(n) = 2c n^{3/2}\]
for the two unique side lengths of each rectangle; $U(n,k)$ and $D(n,k)$ have height $h(n,k)$ and width $w(n,k)$ whereas $L(n,k)$ and $R(n,k)$ have height $w(n,k)$ and width $\ell(n,k)$. The following is a useful equation coming from the Taylor series expansion of $\ell(n,k)$: 
\begin{equation}\label{eq:puiseux}
\begin{split}
    \ell(n,k) = c n^{3/2} \left( (1+k/n)^{3/2} - 1 \right) &= c n^{3/2} \left( 1 + \frac {3}{2} (k/n) + O \left( (k/n)^2 \right) - 1 \right)\\ 
&= \frac{3c}{2} k\sqrt{n} + O\left( \frac{k^2}{\sqrt{n}} \right).
\end{split}
\end{equation}

We are now ready to finish the proof of \cref{lem:3/2main}.

\begin{proof}[Proof of \cref{lem:3/2main}~(a)]
Let $t_i = i^4$ and write $U_i, L_i, R_i, D_i$ as respective shorthands for $U(t_i, t_{i+1}-t_i-1), L(t_i, t_{i+1}-t_i-1), R(t_i, t_{i+1}-t_i-1), D(t_i, t_{i+1}-t_i-1)$. Likewise, write $w_i$ and $\ell_i$ as respective shorthands for $w(t_i, t_{i+1}-t_i-1)$ and $\ell(t_i, t_{i+1}-t_i-1)$. We will burn $G$ in phases, with phase $i$ lasting from turn $t_{i+1}$ to turn $t_{i+2}-1$ and dedicated to burning the four rectangles $U_i, L_i, R_i, D_i$. Note that these four rectangles exist in $G_t$ for all $t \geq t_{i+1}-1$ and so exist during phase $i$. We dedicate an equal number of rounds (up to an additive 1) to burning each rectangle, meaning the number of activator vertices we have available for each rectangle is 
\begin{align*}
\frac{1}{4} \left( t_{i+2}-1 - t_{i+1} \right)
&=
\frac{1}{4} \left( (i+2)^4 - 1 - (i+1)^4 \right)\\
&=
\frac{1}{4} (4i^3 + 18i^2 + 28i + 14)\\
&\geq
i^3 \,,
\end{align*}
and so we may assume we have exactly $i^3$ vertices per rectangle as we can skip turns. 

Let $r_i = \sqrt{w_i \ell_i/2i^3}$ and let $S$ be the set of points defined in Lemma~\ref{lem:rectangle strategy}. With $\tau_i = 4i^3$, by Lemma~\ref{lem:rectangle strategy} we can burn all four rectangles such that
\begin{itemize}
\item on turns $n$ with $t_{i+1} \leq n \leq t_{i+1} + r_i - 1$ the burning balls formed in this rectangle are disjoint, and 
\item on turns $n$ with $n \geq t_{i+1} + r_i + \tau_i$ there is a $1+O\left( \frac{r(w+\ell)}{w \ell} \right)$ fraction of each rectangle that is burning. 
\end{itemize}
Thus, as $t$ increases as the burning game progresses, each set of four rectangles $U_i, L_i, R_i, D_i$ evolves from a new ball  phase (between turns $t_{i+1}$ and $t_{i+1} + r_i - 1$) to a transitional phase (between turns $t_{i+1} + r_i$ and $t_{i+1} + r_i + \tau_i - 1$) and eventually to an old ball phase (from turns $t_{i+1} + r_i + \tau_i$ onward). We are left to show that this evolution, in asymptotics, achieves the bounds given in the proof of Lemma 4.1 (b).

Note that by (\ref{eq:puiseux}), we have that
\begin{align*}
\ell_i &= \frac {3c}{2} (t_{i+1}-t_i-1)t_i^{1/2} + O \left( \frac{(t_{i+1}-t_i-1)^2}{t_i^{1/2}} \right)\\
&= \frac {3c}{2} \left( (i+1)^4-i^4-1 \right) i^2 + O \left( \frac{(i^3)^2}{i^2} \right)\\
&= \frac {3c}{2} (4 i^3) (1 + O(1/i)) i^2 + O(i^4) \\
&= 6c \, i^5 (1 + O(1/i)),
\end{align*}
and $w_i = 2ct_i^{3/2} = 2c \, i^6$. The rectangles $U_i, L_i, R_i, D_i$ are in the new ball phase up to turn
\begin{align*}
t_{i+1} + r_i - 1 
&\geq 
t_i + \left(\frac{w_i\ell_i}{2i^3}\right)^{1/2} - 1\\
&=
t_i + (1 + O(1/i)) \left(\frac{\left(2c \, i^{6}\right) \left( 6 c \, i^5 \right)}{2i^3}\right)^{1/2} - 1 \\
&=
t_i + (1 + O(1/i)) \left( 6 c^2 \, i^8 \right)^{1/2} - 1 \\
&=
t_i + \sqrt{6}c i^4 + O \left( i^3 \right)\\
&\geq 
\left(1 + \sqrt{6}c\right) t_i + O \left( i^3 \right) \,.
\end{align*}

Likewise, the rectangles are in the old ball phase starting on turn
\begin{align*}
t_{i+1} + r_i + \tau_i
&=
t_{i+1} + \sqrt{6}ci^4 + \tau + O(i^3)\\
&\leq
(1+\sqrt{6}c) t_{i+1} + \tau + O(i^3)\\
&=
(1+\sqrt{6}c) t_{i+1} + 4i^3 + O(i^3)\\
&=
(1+\sqrt{6}c) t_i + O(i_3) \,.
\end{align*}
Note that, although balls in a new ball phase rectangle do not overlap with each other, it is possible that overlap occurs between adjacent rectangles. On the other hand, we could modify the tiling in Lemma~\ref{lem:rectangle strategy} by not burning $(r \cdot i, r \cdot j)$ if, say, $\min\{i,j\} < 100$, and the asymptotic result would still hold. Thus, we may ignore the negligible loss coming from the overlap between adjacent rectangles. 

Fix turn $t$ and let $n$ be such that $t = n(1+\sqrt{6}c)$. Next, let $\delta$ be large enough so that $t_i < n-\delta$ implies $U_i, L_i, R_i, D_i$ are old ball phase rectangles and $t_i > n+\delta$ implies they are new ball phase rectangles. By the previous two computations, we can find such a $\delta$ with $\delta = O(t^{3/4})$. The total contribution from old ball phase rectangles is at least
\begin{align*}
4 \sum_{i=1}^{\lfloor(n-\delta)^{1/4}\rfloor} \left(w_i \ell_i + O\left( i^3 \right)\right)
&=
O(n) + 4 \sum_{i=1}^{\lfloor(n-\delta)^{1/4}\rfloor}\left(2ci^6\right) \left(6ci^5\right) (1+O(1/i))\\
&=
O(n) + \bigl(1+O(n^{-1/4})\bigr) \, 4 \sum_{i=1}^{\lfloor(n-\delta)^{1/4}\rfloor}\left(2ci^6\right) \left(6ci^5\right) \\
&=
O(n) + \bigl(1+O(n^{-1/4})\bigr) \, 48c^2 \sum_{i=1}^{\lfloor(n-\delta)^{1/4}\rfloor} i^{11}\\
&=
O(n) + \bigl(1+O(n^{-1/4})\bigr) \, 48c^2 \int_0^{n^{1/4}} x^{11} \, dx\\
&=
O(n) + \bigl(1+O(n^{-1/4})\bigr) \, 4c^2 n^3 \, dx\\
&= 4c^2 n^3 + O(n^{3-1/4})\\
&= 4c^2 n^3 (1+o(1))
\end{align*}
and the total contribution from the new ball phase rectangles is at least
\begin{align*}
\frac{2(t-n-\delta)^3 + t}{3}
&=
\frac{2(\sqrt{6}cn (1+O(n^{-1/4})))^3 + O(n)}{3}\\
&=
\frac{2}{3} (\sqrt{6}cn)^3 + O(n^{3-1/4})\\
&=
\frac{2}{3} (\sqrt{6}cn)^3 (1+o(1)) \,.
\end{align*}
Therefore, the burning density after turn $t$ is at least 
\begin{align*}
\frac{4c^2n^3 + \frac{2}{3}(\sqrt{6}cn)^3}{4f(t)^2} + o(1)
&=
\frac{c^2n^3 + \frac{1}{6}(\sqrt{6}cn)^3}{f(t)^2} + o(1)\\
&=
\frac{c^2n^3(1+\sqrt{6}c)}{c^2n^3(1+\sqrt{6}c)^3} + o(1)\\
&=
\frac{1}{(1+\sqrt{6}c)^2} + o(1) \,,
\end{align*}
and this concludes the proof. 
\end{proof}

\subsection{Smaller densities and the proof of Theorem~\ref{thm:asymptotic}} \label{sec:inbetween}

\begin{proof}[Proof of Theorem~\ref{thm:main} (\ref{case:cubic})]
Let $c > 0$ and let $f = \lceil cn^{3/2} \rceil$. Then $f(n) = \omega(n^2)$, meaning the activator sequence $\cV = (v_n,n \geq 1)$ with $v_n = (0,0)$ for all $n$ achieves density $0$.  

Next, by Lemma~\ref{lem:3/2main}, $(1+\sqrt{6}c)^{-2} \in P(\cG(f))$. Now for $\rho \in (0,(1+\sqrt{6}c)^{-2})$, let $d = \sqrt{\rho (1+\sqrt{6}c)^{2}}$ and let $g(n) = df(n)$. Then again by Lemma~\ref{lem:3/2main} we have $(1+\sqrt{6}c)^{-2} \in P(\cG(g))$, implying by Lemma~\ref{lem:f,g frac density} that $d^2 (1+\sqrt{6}c)^{-2} = \rho \in P(\cG(f))$. 
\end{proof}

Note that this proof finalizes the proof of Theorem~\ref{thm:main}.

\begin{proof}[Proof of Theorem~\ref{thm:asymptotic}]
As discussed previously, Bonato, Gunderson and Shaw proved the strongest possible extension of case $(d)$. 

Let $f : \N \to \N$ be strictly increasing such that $\lim_{n \to \infty} f(n)/n = c$ for some $c \geq 1$. Then for $g(n) = \lceil cn \rceil$ we have $\lim_{n \to \infty} f(n)/g(n) = 1$. By Theorem~\ref{thm:main} (\ref{case:linear}), we have $P(\cG(g)) = [1/(2c^2),1]$. Thus, by Lemma~\ref{lem:f,g same densities}, we also have $P(\cG(f)) = [1/(2c^2),1]$. 

Finally, let $f : \N \to \N$ be strictly increasing such that $\lim_{n \to \infty} f(n)/n^{3/2} \to c$ for some $c > 0$. Then, for $g(n) = \lceil cn^{3/2} \rceil$ we have $\lim_{n \to \infty} f(n)/g(n) = 1$. By Theorem~\ref{thm:main} (\ref{case:cubic}), we have $P(\cG(g)) = [0,(1+\sqrt{6}c)^{-2}]$. Thus, by Lemma~\ref{lem:f,g same densities}, we also have $P(\cG(f)) = [0,(1+\sqrt{6}c)^{-2}]$. 
\end{proof}

\section{Further Directions}\label{sec:conclusion}

A natural extension of this work is to study the attainable densities on growing $d$-dimensional grids. To this end, we pose the following conjecture.

\begin{conjecture}
Let $f : \N \to \N$ be a strictly increasing function, let $G_n = [-f(n),f(n)]^d$ for some $d>2$, and let $\cG(f) = (G_n,n \geq 1)$. 
\begin{enumerate}
\item[(a)] If $f(n) = \lceil cn \rceil$ for some $c \geq 1$ then $P(\cG(f)) = [1/(2c^d),1]$.
\item[(b)] If $f(n) = \omega(n)$ and $f(n) = o(n^{(d+1)/d})$ then $P(\cG(f)) = [0,1]$.
\item[(c)] If $f(n) = \lceil cn^{(d+1)/d} \rceil$ for some $c > 0$ then $P(\cG(f)) = [0,\phi(c,d)]$ for some constant $\phi(c,d)$ depending on $c$ and $d$.
\item[(d)] If $f(n) = \omega(n^{(d+1)/d})$ then $P(\cG(f)) = \{0\}$.  
\end{enumerate}
\end{conjecture}
\noindent
In a recent work by Blanc and Contat~\cite{BC25}, they study random burning on the $d$-dimensional Euclidean lattice, and their results could give insights as to the correct function $\phi(c, d)$ in case $(c)$. Moreover, their generalization of Theorem~\ref{thm:MPR} from Mitsche, Pra\l{}at and Roshanbin~\cite{MPR17} might allow for a generalization of our proof of Theorem~\ref{thm:main} (b). 

Another direction for future work is to find necessary and sufficient growth conditions for the function $f(n)$ so that $P(\cG(f)) \neq \emptyset$, or so that $\rho \in P(\cG(f))$ for some $\rho \in (0, 1)$. We showed via example that at least \textit{some} restriction on the growth of $f$ is necessary to obtain a burning density on $\cG(f)$, though we do not claim that controlled growth is the weakest possible restriction.

\end{document}